\documentclass[11pt,letterpaper]{article}

\usepackage{amsmath}
\usepackage{amsbsy}
\usepackage{amssymb}
\usepackage{euscript}
\usepackage{tabularx}
\usepackage{comment}
\usepackage{amsthm}
\usepackage{graphicx}

\theoremstyle{plain}
\newtheorem{theorem}{Theorem}
\newtheorem{lemma}{Lemma}[section]
\newtheorem{prop}{Proposition}[section]
\theoremstyle{remark}
\newtheorem{exa}{Example}[section]
\newtheorem{remark}{Remark}
\newtheorem{ques}{Question}

\newcommand{\cS}{\mathcal{S}}

\newcommand{\cM}{\mathcal{M}}
\newcommand{\RR}{{\mathbb{R}}}

\newcommand{\WW}{{\mathcal{W}}}
\newcommand{\RRx}{\RR\langle x \rangle}
\newcommand{\gtupn}{(\RR^{n\times n}_{sym})^g}

\def\dd2{\frac{d}{2}}
\def\dm2{\frac{d-1}{2}}
\def\ll2{\frac{\ell}{2}}

\newcommand{\ncm}{\,}

\newcommand{\eeq}{\notag}
\newcommand{\summ}[2]{\underset{#1}{\overset{#2}{\sum}}}
\newcommand{\dird}{D} 
\DeclareMathOperator{\nchess}{NCHes}

\newcommand{\lap}{Lap}

\DeclareMathOperator{\re}{Re}
\DeclareMathOperator{\im}{Im}

\newcommand{\h}{\gamma}
\newcommand{\hd}{\h^d}
\newcommand{\hdr}{\re(\h^d)}
\newcommand{\hdi}{{\im(\h^d)}}
\newcommand{\hdl}{{\h^{d-1}}}
\newcommand{\hdrl}{{\re(\h^{d-1})}}
\newcommand{\hdil}{{\im(\h^{d-1})}}

\newcommand{\ncder}[3]{\dird[#1,#2,#3]}
\newcommand{\ncl}[2]{\lap[#1,#2]}

\title{Non-Commutative Harmonic and Subharmonic Polynomials}
\author{J. William Helton, Daniel P. McAllaster and
Joshua A. Hernandez \\ \\
UCSD Department of Mathematics \\
La Jolla, CA. 92093\\ \\
}

\begin{document}

\maketitle

\begin{abstract}
The paper introduces a notion of the Laplace operator of a polynomial
$p$ in noncommutative variables $x=(x_1, \ldots, x_g)$.  The Laplacian
$\lap[p,h]$ of $p$ is a polynomial in $x$ and in a noncommuting
variable $h$. When all variables commute we have $\lap[p,h]= h^2
\Delta_x p$ where $\Delta_x p$ is the usual Laplacian.  A symmetric
polynomial in symmetric variables will be called {\it harmonic} if
$\lap[p,h]=0$ and {\it subharmonic} if the polynomial $q(x,h):=
\lap[p,h]$ takes positive semidefinite matrix values whenever matrices
$X_1, \ldots, X_g, H$ are substituted for the variables $x_1, \ldots,
x_g, h$. In this paper we classify all homogeneous symmetric harmonic
and subharmonic polynomials in two symmetric variables. We find there
are not many of them: for example, the span of all such subharmonics
of any degree higher than 4 has dimension 2 (if odd degree) and 3 (if
even degree). Hopefully, the approach here will suggest ways of
defining and analyzing other partial differential equations and
inequalities.
\end{abstract}

\newpage

\section{Introduction}

In the introduction we shall make essential definitions, then state
our main results. The rest of the paper proves them.

\subsection{Definitions}

\subsubsection{Non-Commutative Polynomials}

A non-commutative monomial $m$ of degree $d$ on the free variables
$(x_1,\,\ldots,\,x_g)$ is a product $x_{a_1}x_{a_2}\cdots x_{a_d}$ of
these variables, corresponding to a unique sequence of $a_i$ of
nonnegative integers, $1\le a_i\le g$.  We abbreviate this $m = x^w$,
where $w$ is the $d$-tuple $(a_1,\ldots,a_d)$. The set of all
monomials in $(x_1,\,\ldots,\,x_g)$ is denoted as $\cM$ and the set of
indexes $w$ is denoted $\WW$.
Some notation is:

\begin{center}
\begin{tabular}{ll}
$|w|=d$& the length of $w$\\
$(w)_i=a_i$ & the $i^{\text{th}}$ entry of $w$ \\
$w^T=(a_d,\ldots,a_1)$& the transpose of $w$ \\
$\phi= ()$ & the empty word (word of length zero)
\end{tabular}
\end{center}

The space of non-commutative polynomials $p(x) = p(x_1,\ldots,x_g)$
with real coefficients is denoted $\RRx$ and we express $p$ as
\begin{equation}
p(x)  = \summ{m\,\in\, \cM}{}A_{m}\,m.
\eeq\end{equation}
An example of a non-commutative polynomial is \
$$p(x) = p(x_1,x_2)
 = x_1^2\ncm x_2\ncm x_1 + x_1\ncm x_2\ncm x_1^2 + x_1
\ncm x_2 - x_2\ncm x_1 + 7$$ \
(in commutative variables, this would be equivalent to $2x_1^{3}x_2 + 7$).

The transpose of a monomial $m=x^w$ is defined to be $m^T=x^{w^T}$.
The transpose of a polynomial $p$, denoted $p^T$, is defined by
$p(x)  = \summ{m\,\in\, \cM}{}A_{m}\,m^T$
and has the following properties:
\begin{quote}
(1) $(p^T)^T = p$ \\
(2) $(p_1 + p_2)^T = p_1^T + p_2^T$ \\
(3) $(\alpha p)^T = \alpha p^T$ \qquad\qquad ($\alpha \in
\RR$)\\
(4) $(p_1\ncm p_2)^T = p_2^T\ncm p_1^T$.
\end{quote}
In this paper, we shall consider primarily
 polynomials in symmetric variables.
That is, we consider variables $x_i$ where $x_i^T = x_i$.
Then monomials satisfy
\
$(x_{a_1}\ncm \ldots\ncm x_{a_{d-1}})^T
=
x_{a_{d-1}}\ncm\ldots\ncm x_{a_1} $,
which in other notation is
$(x^w)^T = x^{w^T}.$
Symmetric (or self-adjoint) polynomials are those that are equal to their
transposes.

\subsubsection{Evaluating Noncommutative Polynomials}
\label{sec:openG2}

Let $\gtupn$ denote the set of $g$-tuples
$(X_1,\ldots,X_g)$ of real symmetric $n\times n$ matrices.
We shall be interested in evaluating a
polynomial $p(x)=p(x_1,\ldots,x_g)$
that belongs to $\RRx$ at a
tuple $X=(X_1,\dots,X_g)\in(\mathbb{R}^{n\times n}_{sym})^g$.
In this
case $p(X)$ is also an $n\times n$ matrix and
the involution on $\RRx$
that was introduced earlier is compatible
with matrix transposition, i.e.,
$$
p^T(X)=p(X)^T,
$$
where $p(X)^T$ denotes the transpose of the  matrix  $ p(X)$.
When $X\in \gtupn$ is substituted into $p$
the constant term  $p(0)$ of $p(x)$  becomes $p(0) I_n$.
  Thus,
for example,
   $$p(x)= 3+x_1^2+5x_2^3  \  \ \implies \ \  p(X)= 3 I_n+X_1^2+5X_2^3.$$

A symmetric polynomial $p \in \RRx$ is {\bf matrix positive}
if $p(X)$ is a positive semidefinite matrix for each tuple
$X=(X_1,\dots,X_g)\in (\mathbb{R}^{n\times n}_{sym})^g$.
We emphasize that throughout this paper, unless otherwise noted,
$x_1, x_2, \ldots, x_n$ stand for variables and
$X_1, X_2, \ldots, X_n$ stand for matrices (usually symmetric).

\subsubsection{Non-Commutative Differentiation}

For our non-commutative purposes,
we take directional derivatives in $x_i$
with regard to an indeterminate direction parameter $h$.
\begin{equation}
\dird[p(x_1,\ldots,x_g),x_i,h]:= \frac{d}{dt}[p(x_1, \ldots, (x_i+th),
\ldots, x_g)]_{|_{t=0}}.
\end{equation}
We say that this is the directional derivative of
$p(x) = p(x_1,\ldots,x_g)$
in $x_i$ in the direction $h$.
Note it is linear in  $h$.
For a detailed formal definition see \cite{HMVjfa},
for more examples see \cite{CHSY03}.
\begin{exa} The directional derivative
\end{exa}
\begin{tabularx}{\linewidth}{X}
$\dird[x_1^2\ncm x_2,x_1,h] = \frac{d}{dt}[(x_1+th)^2x_2]_{|_{t=0}}$\\
\quad\qquad\qquad\qquad$= \frac{d}{dt}[x_1^2\ncm x_2+th\ncm x_1\ncm
x_2+tx_1\ncm h\ncm x_2+t^2h^2\ncm x_2]_{|_{t=0}}$\\
\quad\qquad\qquad\qquad$
= [h\ncm x_1\ncm x_2 + x_1\ncm h\ncm x_2 + 2th^2\ncm
x_2]_{|_{t=0}}$\\
\quad\qquad\qquad\qquad$= h\ncm x_1\ncm x_2 + x_1\ncm h\ncm x_2$.\\
\end{tabularx}

As this example shows, the directional derivative of $p$ on $x_i$
in the direction $h$ is the sum of the terms produced by replacing
one instance of \mbox{$x_i$ with $h$.}

\begin{lemma}
The directional derivative
of  NC polynomials is linear,
\begin{equation}
\dird[a\,p(x) + b\,q(x),x_i,h]
= a\,\dird[p(x),x_i,h] + b\,\dird[q(x),x_i,h]
\eeq\end{equation}
and respects transposes
\begin{equation}
\dird[p(x)^T,x_i,h]
= \dird[p(x),x_i,h]^T.
\eeq\end{equation}
\end{lemma}

\proof \ Straighforward. \qed

\subsubsection{Non-Commutative  Laplacian and Subharmonicity}

  The Laplacian of a NC polynomial $p(x)$ is
 defined as:
\begin{align}
\lap[p, h] &:= \summ{i=1}{g}\dird[\dird[p(x),
x_i,h],\ x_i,h] \label{eq:dirdy}\\
&=
\summ{i=1}{g}\frac{d^2}{dt^2}
[p(x_1,\ldots,(x_i+th),\ldots,x_g)]_{|_{t=0}}.
\end{align}
Our notation is slightly inconsistent (but has advantages)
in that the single letter $x$
stands for $g$ variables  $x_1, \ldots, x_g$
while $h$ is a single variable.
Note that $\lap$ is linear in $h$.
An NC  polynomial is called {\bf harmonic}
if its Laplacian is zero, and {\bf subharmonic}
if its Laplacian is matrix-positive and
{\bf purely subharmonic} \index{purely subharmonic}
 is used to describe a polynomial
which is subharmonic but not harmonic -
that is, having a nonzero, matrix-positive Laplacian.

Specialization of  $\lap[p,h]$,
to commutative variables, is
$h^2
\Delta \bigl(p \bigr)
$
where $\Delta \bigl(p\bigr)$ is the standard Laplacian, namely,
$
\Delta \bigl(p \bigr) := \summ{i=1}{g}\partial_{x_ix_i}p(x).
$
Here $p: \RR^n \to \RR$.

\subsection{Classification of Harmonics and Subharmonics in Two Variables}

For our special homogeneous polynomials on two variables, define
\begin{gather}\label{eq:gammadef}
\h := x_1+i\,x_2
\end{gather}
where $i$ is the imaginary number.

\begin{theorem}
\label{thm:main}

The homogeneous noncommutative polynomials in two
symmetric variables which are
\begin{description}
\item[(1a.)] \ harmonic of degree $d>2$
are exactly the linear combinations of
\begin{equation}
\re(\hd) \quad \text{and} \quad \im(\hd),
\eeq\end{equation}

\item[(1b.)] \ harmonic of degree $d=2$
are exactly the linear combinations of
\begin{equation}
\re(\hd) \quad \text{and} \quad \im(\hd) \quad and \quad x_1 x_2,
\eeq\end{equation}
\hspace{.5in}
(note this includes $x_2x_1$),

\item[(2a)] \ subharmonic  of degree $2d$ with $d>2$ ,
are exactly the linear combinations:
\begin{align}
c_0& [\re(\h^d)]^2\, +\, c_1\re(\h^{2d})\, +\, c_2\im(\h^{2d})  \notag \\
&= c_0 [\im(\h^d)]^2\,+\,(c_0+c_1)\,\re(\h^{2d})\,+\,c_2\im(\h^{2d})
\label{eq:lincomb}
\end{align}
\hspace{.5in} where $c_0 \geq 0$,
\vspace{1ex}

\item[(2b)] \ symmetric  subharmonics  of degree $4$,
are exactly the linear combinations:
\begin{eqnarray}
f&=&B_1(x_1^4-x_1^2\ncm x_2^2-x_2^2\ncm x_1^2+x_2^4)
+B_2(x_1^3\ncm x_2+x_2\ncm x_1^3-x_2\ncm x_1\ncm x_2^2-x_2^2
\ncm x_1\ncm x_2)\\
&&+B_3(x_1^2\ncm x_2\ncm x_1 + x_1\ncm x_2\ncm x_1^2 - x_1
\ncm x_2^3-x_2^3\ncm x_1 )
+B_4(x_1\ncm x_2\ncm x_1\ncm x_2+x_2\ncm x_1\ncm x_2\ncm x_1)\notag\\
&&+B_5\,x_1\ncm x_2^2\ncm x_1 + B_6 x_2\ncm x_1^2\ncm x_2 \notag
\end{eqnarray}
with coefficients satisfying the inequalities:
\begin{align}
&(III)\implies(B_1 + B_6)(B_1 + B_5) >
(B_3 - B_2)^2 + (B_1 + B_4)^2 \text{ and }\\
&(I)\implies\quad B_1 + B_6> 0\quad (\text{or, equivalently}
\quad B_1 + B_5 >
0).\end{align}
\item[(2c)] \ All  subharmonics  of degree $2$ are,
$$
 A_1 \, x_1^2 + A_2 \, x_2^2  + A_3 \,x_1 x_2 + A_4 x_2 x_1
$$
\hspace{.5in}
with $A_1+A_2 \geq 0$.

\item[(3)] \ Pure subharmonics of odd degree do not exist.
\end{description}
\vspace{1ex}

Note:
all of these functions except for $x_1 x_2$
and $x_2 x_1$ in (1b) and in (2c)
are symmetric.
\end{theorem}
\proof \
Most of the remainder of this paper is focused on proving this theorem.
The proofs for the parts of the theorem are as follows:

\begin{center}

\begin{tabular}{|c|c|}\hline
Part of theorem&Section of the proof\\\hline\hline
(1a.)&\ref{sec:main:1a}\\\hline
(1b.)&\ref{sec:main:1b}\\\hline
(2a.)&\ref{sec:main:2a}\\\hline
(2b.)&\ref{sec:main:2b} \ also Remark \ref{rem:deg6} \\\hline
(2c.)&\ref{sec:main:2c}\\\hline
(3.)&\ref{sec:main:3}\\\hline
\end{tabular}

\end{center}

\medskip

\begin{remark}
The following degree 3 polynomial $p$
is unusual in that there is a region of $X_1,X_2$ where
$\lap(p)$ is positive, but $\lap(p)$ is not positive everywhere:
$$A_1\, (x_1^3 - x_1\ncm x_2^2 - x_2\ncm x_1^2)
+ A_2\, x_2\ncm x_1\ncm x_2 + A_3\,\ncm x_1\ncm x_2\ncm x_1
+ A_4\, (x_2^3 - x_1^2\ncm x_2 - x_2\ncm x_1^2)$$
For this the
\textit{region of subharmonicity} is
$(A_1\, + A_2)x_1 + (A_3 + A_4)x_2 > 0\,$  \ and
the
\textit{region of harmonicity} is
$A_1\, + A_2 = A_3 + A_4 = 0$.
Of course,
there is no homogeneous polynomial of degree three which is
  subharmonic over all values of $x_1$ and $x_2$.
\end{remark}

\subsection{Subharmonics are All Built from Harmonics}

Our second main result is
 a general fact which holds in any number of variables:
\begin{theorem}
\label{thm:SubsToSquaresOfHarms}
Assume the harmonic polynomials
homogeneous of degree $\dd2$
have a basis $\h_1, \cdots, \h_k$
with the independence property:
 there is a monomial $w_j$ in $\h_j$ which does not occur
in the other $\h_1,  \cdots, \h_k$. \
  If $p$ is a homogeneous symmetric subharmonic polynomial
 of degree even $d$, then $p$ has the form
 $$ p=  \sum_i^{finite} c_i R_i^TR_i $$
for some homogeneous harmonic functions $R_j$
of degree $\dd2$ and real numbers $c_j$.
\end{theorem}
Because of this,
 knowing all homogeneous subharmonics will
likely occur once the harmonics are classified.

\proof \ The proof is found in \S \ref{sec:EvenSumHarms}.

\subsection{Comparison with Commutative Subharmonic Polynomials}

The study of harmonic and subharmonic polynomials in commuting
variables is classical.
Harmonic commuting polynomials are classified in any number
of variables and the have a close correspondence to
 spherical harmonics.
 A good reference on this is
 \cite{HoweTan} \S 2.4, pp. 110-113.

For two commuting  variables,
the homogeneous  harmonic polynomials are
 those of the form,
$$ Re (x_1 + I x_2)^n \ \ \  and  \ \ \  Im (x_1 + I x_2)^n,$$
so the commuting and noncommuting case are exactly parallel.

\subsection{Related Topics and Motivation}

\noindent
{\bf Non-Commutative Convexity} \
The non-commutative Hessian is defined as:
\begin{equation*}
\nchess[p(x_1,\ldots,x_g), \{ x_1,\eta_1\},\ldots, \{ x_g,\eta_g\}] :=
\frac{d^2}{dt^2}[p(x_1+t\eta_1,\ldots,x_g+t_g\eta_g)]_{|_{t=0}}. \\
\end{equation*}
Note that this is composed of several independent direction parameters,
$\eta_i$ and that if $p$ is a polynomial, then its Hessian is a polynomial
in $x$ and $\eta$ which is homogeneous of degree 2 in $\eta$.

A non-commutative polynomial is considered {\bf convex} wherever  its
Hessian is  matrix-positive.

A polynomial $p(x) = p(x_1,\ldots,x_d)$ is {\bf geometrically convex}
 if and only if, for every $X, Y \in \gtupn  $,
\begin{equation}
\frac{1}{2}\bigl(p(X) + p(Y)\bigr) - p\biggl(\frac{X + Y}{2}\biggr)
\eeq\end{equation}
is positive-semidefinite.
It is proved in \cite{geom}
that convexity is equivalent to geometric convexity.
A crucial fact regarding these polynomials
(see \cite{ConvPoly}) is that they are
all of degree two or less.
Some excellent papers on noncommutative convexity are
\cite{HanT06} \cite{Han97}.

The commutative analog of this ``directional''
Hessian is the quadratic function
\begin{equation}
H\bigl(p\bigr)
\begin{pmatrix} \eta_1 \\ \vdots \\ \eta_g
\end{pmatrix} \cdot
\begin{pmatrix} \eta_1 \\ \vdots \\ \eta_g
\end{pmatrix}
\end{equation}
where $H\bigl(p \bigr)$ is the Hessian matrix:
\begin{equation}
\begin{pmatrix}
\partial_{x_1x_1}p(x) &\cdots &\partial_{x_1x_g}p(x) \\
\vdots & \ddots & \vdots \\
\partial_{x_gx_1}p(x) &\cdots &\partial_{x_gx_g}p(x)
\end{pmatrix}.
\end{equation}
If this Hessian is positive
semidefinite  at all
 $(x_1,\ldots,x_g)$,
then $f$ is said to be convex.

\bigskip

\noindent
{\bf Non-Commutative Algebra in Engineering} \
Inequalities, involving polynomials in matrices and their
inverses, and associated optimization problems have become very
important in engineering.
When such polynomials are  matrix convex,
local minima are global. This is extremely important
in applications.
Also,
interior point numerical  methods apply well to these.
In the last few
years, the approaches that have been proposed in the field of
optimization and control theory based on linear matrix
inequalities and semidefinite programming have become very
important and promising, since the same framework can be used for
a large set of problems. Matrix inequalities provide a nice setup
for many engineering and related problems, and if they are convex
the optimization problem is well behaved and interior point
methods provide efficient algorithms which are effective on
moderate sized problems.
Unfortunately, the class of matrix convex
noncommutative polynomials  is very small;
as already mentioned
they are all of degree two or less \cite{ConvPoly}.

Our original interest in subharmonic polynomials
 was to analyze conditions similar to
convexity, though not as restrictive, in the hopes of finding much
broader classes which still had nice properties.
What we found (as reported here)
was that subharmonic polynomials are (in two) variables
a highly restricted class.

\bigskip

\noindent {\bf Noncommutative Analysis} This article would come
under the general heading of ``free analysis",
since the setting is a noncommutative algebra whose generators
are ``free" of relations. This is a
burdgeoning area, of which free probability is currently the
largest component. The interested reader is referred to the web
site \cite{AIMfree06}
of American Institute of Mathematics, in particular it gives the
findings of the AIM workshop in 2006 on free analysis.

A fairly expository article describing noncommutative convexity,
noncommutative semialgebraic geometry and relations to engineering is
\cite{HPprept}.

\section{Existence Proofs}
Now we set about to prove Theorem \ref{thm:main}. In this section, we show that the polynomials claimed to be harmonic and subharmonic are indeed. In section \ref{sec:unique}, we show that these are the only posibilities.
\subsection{Product Rules for Derivatives}
To begin with, we will build up facts about derivatives.
\subsubsection{Product Rule for First Derivatives}
\begin{lemma}
The product rule for the  directional derivative of
NC polynomials
is
\begin{equation}
\label{eq:DProd}
\dird[p_1\ncm {p_2},x_i,h]
= \dird[{p_1},x_i,h]\ncm {p_2}\, +\, {p_1}\ncm \dird[{p_2},x_i,h].
\eeq\end{equation}
\end{lemma}

\proof \ \
The directional derivative $\dird[m,x_i,h]$
of a product $m=m_1m_2$ of non-commutative monomials
$m_1$ and $m_2$ is the sum of terms produced by
replacing one instance of $x_i$ in $m$ by $h$.
This sum can be divided into two parts:

$\mu_1$, the sum of terms whose $h$ lie in the first $|m_1|$ letters,
\ {i.e.} \ \ $ \dird[m_1,x_i,h]m_2$

$\mu_2$, the sum of terms whose $h$ lie in the last $|m_2 |$ letters,
\ {i.e.} \
$\ m_1 \dird[m_2,x_i,h].
$

\noindent
Therefore
\begin{equation}
\dird[m_1m_2,x_i,h] = \dird[m_1,x_i,h]m_2 + m_1\dird[m_2,x_i,h].
\eeq\end{equation}

We can extend this product rule to the product
of any two non-commutative polynomials $p_1$ and $p_2$ as follows.
\begin{align}
\dird&[p_1\ncm p_2,x_i,h] = \dird\Bigl[\Bigl(\underset{m_1\in
\WW_{p_1}}{\sum}A_{m_1}m_1\Bigr)\ncm\Bigl(\underset{m_2\in
\WW_{p_2}}{\sum}A_{m_2}m_2\Bigr),x_i,h\Bigr]  \notag\\
&= \underset{m_1\in \WW_{p_1}}{\sum}\;\underset{m_2\in \WW_{p_2}}{\sum}
A_{m_1}A_{m_2}\dird[m_1\ncm m_2,x_i,h] \notag\\
&= \underset{m_1\in \WW_{p_1}}{\sum}\;\underset{m_2\in
\WW_{p_2}}{\sum}A_{m_1}A_{m_2}\dird[m_1,x_i,h]\ncm m_2 +
A_{m_1}A_{m_2}\,m_1\ncm \dird[m_2x_i,h]\notag \\
&= \dird\Bigl[\!\!\underset{m_1\in
\WW_{p_1}}{\sum}\!\!A_{m_1}m_1,x_i,h\Bigr]\ncm \!\!\underset{m_2\in
\WW_{p_2}}{\sum}\!\!A_{m_2}m_2 + \!\!\underset{m_1\in
\WW_{p_1}}{\sum}\!\!A_{m_1}m_1\ncm \dird\Bigl[\!\!\underset{m_2\in
\WW_{p_2}}{\sum}\!\!A_{m_2}m_2,x_i,h\Bigr]\notag\\
&= \dird[p_1,x_i,h]\ncm p_2 + p_1\ncm \dird[p_2,x_i,h].
\end{align}\qed

\subsubsection{The Laplacian of a Product}
We now prove Theorem \ref{thm:main} part (2b).
\begin{lemma}
\label{lem:prodRulap}
The product rule for the Laplacian of NC polynomials is
\begin{equation}
\lap[p_1 \, p_2,h] = \lap[p_1,h]\, p_2 + p_1\ncm \lap[p_2,h]
     + 2\summ{i=1}{g}\bigl( \dird[p_1,x_i,h]\ncm\dird[p_2,x_i,h] \bigr).
\eeq\end{equation}
As a consequence if $p$ is harmonic, then
\begin{equation}
\label{eq:lapSq}
\lap[p^T \, p,h] =
      2\summ{i=1}{g}\bigl( \dird[p,x_i,h]^T \ncm\dird[p,x_i,h] \bigr).
\eeq\end{equation}
\end{lemma}

\proof
\begin{align*}
\label{eq:LapPros}
\lap[p_1 \, p_2,h]
& = \summ{i=1}{g} \dird[\dird[p_1\,p_2,x_i,h],x_i,h] \\
&= \summ{i=1}{g} \bigl( \dird[p_1\,\dird[p_2,x_i,h] +
      \dird[p_1,x_i,h]\,p_2, \ x_i,h]\bigr) \\
&= \summ{i=1}{g} \bigl(p_1\,\dird[\dird[p_2,x_i,h],x_i,h] +
     \dird[\dird[p_1,x_i,h],x_i,h]\,p_2 \\
& \qquad \qquad + 2\dird[p_1,x_i,h],\dird[p_2,x_i,h]\bigr)\\
&=\lap[p_1,h]\, p_2 + p_1\ncm \lap[p_2,h]
     + 2\summ{i=1}{g}\bigl( \dird[p_1,x_i,h]\ncm\dird[p_2,x_i,h] \bigr).
\end{align*}
\qed

\subsection{Formulas Involving  $\h^d$ and its Derivatives}

Recall from (\ref{eq:gammadef}) that $\h := x_1 + ix_2.$

 Note that $\h = \h^T$ and therefore that $\h^d = (\h^d)^T$.
So
\begin{equation}
\left(\re(\hd)\right)^T = \re((\h^d)^T) =
\re(\h^d)\text{ and }\left(\im(\hd)\right)^T=\im((\h^d)^T) = \im(\h^d).
\eeq\end{equation}
This proves the (last) assertion in Theorem \ref{thm:main}
that all but a few subharmonics on our list are symmetric.

\begin{lemma}
\label{lem:sym}
The derivatives of of $\h^d$
exhibit the following symmetries.
\begin{equation}
\dird[ \hdr,x_1,h] = \dird[ \hdi,x_2,h]
\eeq\end{equation}
and
\begin{equation}
\dird[ \hdr,x_2,h] = -\dird[\hdi,x_1,h].
\eeq\end{equation}
\end{lemma}

\proof\ The proof proceeds by induction. To begin, it is easily seen that
\begin{align*}
\dird[\re(\h),x_1,h] &= \dird[\im(\h),x_2,h] = h \\
\dird[\im(\h),x_1,h] &= -\dird[\re(\h),x_2,h] = 0.
\end{align*}
Assume that
\begin{align*}
\dird[\hdrl,x_1,h] &= \dird[\hdil,x_2,h]\\
\dird[\hdil,x_1,1] &= -\dird[\hdrl,x_2,h].
\end{align*}
Then
\begin{align*}
\dird[&\hdr,x_1,h] = \dird[x_1\hdrl - x_2\hdil,x_1,h]\\
&= x_1\,\dird[\hdrl,x_1,h] + h\,\hdrl - x_2\,\dird[\hdil,x_1,h] \\ \\
\dird[&\hdi,x_2,h] = \dird[x_1\hdil + x_2\hdrl,x_2,h]\\
&= x_1\,\dird[\hdil,x_2,h] + x_2\,\dird[\hdrl,x_2,h] + h\,\hdrl,
\end{align*}
so
\begin{equation}
\dird[ \hdr,x_1,h] = \dird[ \hdi,x_2,h] \eeq\end{equation}
which satisfies the first half of our inductive hypothesis.
For the next half compute
\begin{align*}
\dird[&\hdr,x_2,h] = \dird[x_1\hdrl - x_2\hdil,x_2,h]\\
&= x_1\,\dird[\hdrl,x_2,h] - x_2\,\dird[\hdil,x_2,h] - h\,\hdil \\ \\
\dird[&\hdi,x_1,h] = \dird[x_1\hdil + x_2\hdrl,x_1,h]\\
&= x_1\,\dird[\hdil,x_1,h] + h\,\hdil + x_2\,\dird[\hdrl,x_2,h],
\end{align*}
so \begin{equation} \dird[ \hdr,x_2,h] = -\dird[\hdi,x_1,h]. \eeq\end{equation}
\qed

\subsection{Harmonics $degree \ >2$: Proof of Theorem
\ref{thm:main} \mbox{part (1)}}

Our proof will proceed by induction. Since the Laplacian of words of length 1 is zero,
\begin{equation}
\lap[\re(\h),h] = \lap[x_1,h] = 0 \quad \text{and} \quad  \lap[\im(\h),h] =
\lap[x_2,h] = 0 \label{eq:length1}.
\eeq\end{equation}
Now, assume that
\begin{equation}
\lap[\hdrl,h] = \lap[\hdil,h] = 0 \label{eq:indyp}.
\end{equation}
Pushing ahead,
\begin{align}
&\hdr = \re((x_1+ix_2)\,\hdl) = x_1\ncm\hdrl -
x_2\ncm\hdil\label{eq:first}\\
&\hdi = \im((x_1+ix_2)\,\hdl) = x_1\ncm\hdil + x_2\ncm\hdrl\label{eq:last}.
\end{align}
Applying our product rule to (\ref{eq:first}):
\begin{align*}
\lap[\hdr,h] &= \lap[x_1\ncm\hdrl,h] - \lap[x_2\ncm\hdil,h] \\
&= x_1\ncm\lap[\hdrl,h] + \lap[x_1,h]\ncm\hdrl \\
&\;\;\;\;\;\; + 2\dird[x_1,x_1,h]\ncm\dird[\hdrl,x_1,h]  \\
&\;\;\;\;\;\; + 2\dird[x_1,x_2,h]\ncm\dird[\hdrl,x_2,h]\\
&\;\;\;\; - x_2\ncm\lap[\hdil,h] - \lap[x_2,h]\ncm\hdil \\
&\;\;\;\;\;\; - 2\dird[x_2,x_1,h]\ncm\dird[\hdil,x_1,h]\\
&\;\;\;\;\;\; - 2\dird[x_2,x_2,h]\ncm\dird[\hdil,x_2,h].\\
\end{align*}
Use (\ref{eq:length1}) and (\ref{eq:indyp}) to obtain that the $\lap[]$
terms are 0, and that ``cross partials are 0" to get
$$
\lap[\hdr,h] = h\ncm\dird[\hdrl,x_1,h] - h\ncm\dird[\hdil,x_2,h].
$$
By symmetry Lemma \ref{lem:sym},
this means
\begin{equation}
\re(\lap[\hd,h]) = \lap[\hdr,h] = 0.
\eeq\end{equation}
By a similar argument, applying the product rule to  (\ref{eq:last}),
\begin{equation}
\im(\lap[\hd,h]) = \lap[\hdi,h] = 0.
\eeq\end{equation}
\qed

\subsection{Subharmonics $degree \ >4$:
Proof of Theorem \mbox{\ref{thm:main}  (2a.)}\label{sec:main:2a}}
\label{sec:main:3}

The product rule for the Laplacian of harmonics in
Lemma \ref{lem:prodRulap} says
$
\lap[\left(\hdr\right)^2,h]
$
is a sum of squares.
Thus we have shown $\left(\hdr\right)^2$ is subharmonic.

Now we prove
the formula (\ref{eq:lincomb})
relating subharmonics.
We use
\begin{align*}
\h^{2d} &= (\re(\h^d) + i\im(\h^d))^2\\
&= (\re(\h^d))^2 - (\im(\h^d))^2 + i(\re(\hd)\hdi + \hdi\hdr).
\end{align*}
Therefore
\begin{equation}
\re(\h^{2d}) = \left(\hdr\right)^2 - \left(\hdi\right)^2.
\eeq\end{equation}
So
\begin{align*}
c_0 \left[\re(\h^d)\right]^2\, +\,  &c_1\re(\h^{2d})\,+\, c_2\im(\h^{2d}) \\
&= c_0 \left[\im(\h^d)\right]^2\,+\,(c_0+c_1)\,(\re(\h^{2d})) +
c_2\im(\h^{2d}),\end{align*}
which is (\ref{eq:lincomb}).\qed

Up to this point we have handled subharmonics of even degree.
To see that there are no pure subharmonics of odd degree,
note that the Laplacian $L(x)$ of an odd degree
polynomial is itself an odd degree polynomial which is matrix-positive.
Consider $L(t x)$ as $t \in \RR$ approaches $\pm \infty$.
Since the highest order terms dominate,
the signs of these limits are opposite.
Thus the highest order terms are 0.

\section{Classification when Degree is Four or Less}
We handle now what appear to be special cases
which are exceptions to the general degree $>4$
theorem.

\subsection{The Matrix Representation}

Important in our proofs for polynomials of low degree
is a representation of polynomials\\$q(x_1, \cdots, x_g)[h]$
which are homogeneous of degree 2 in $h$.
Recall that often
$x$ stands for $ (x_1, \cdots, x_g)$ and $h$ is a single variable.
In our notation $q(x)[h]$ we use $[ \ ]$ to distinguish the
variable which is of degree 2.

Any NC symmetric polynomial $q$ in symmetric variables
 quadratic in $h$
can be  written
\begin{equation} q(x)[h]
= \summ{i = 1}{n}\summ{j=1}{n}
\,(h\,{m_i})^T\,{  Z_{ij}(x) }\,(h\,{m_j})
= \summ{i =
1}{n}\summ{j=1}{n}
\,({m_i}^T \, h) \,{Z_{ij}}\,(h\,{m_j})
\eeq\end{equation}
where $m_i$ are monomials in $x$
 and $Z_{ij}(x)$ are polynomials in $x$.

Define $Z(x)$ as the $N$-by-$N$ matrix of polynomials in $x$
whose $i,j^{\text{th}}$ element is
$Z_{ij}$, and define $V(x)[h]$ as
\begin{equation}
V(x)[h]^T = h(m_1,\,m_2,\,\ldots,\,m_N).
\eeq\end{equation}
We call $Z$ the {\bf middle matrix}
for $q$ and $V$ its {\bf border vector}.
In this notation our representation is
$$ q(x,h)= V(x)[h]^T Z(x) V(x)[h]$$
We can and typically do take $Z(x)$ to be symmetric.
If the monomials $m_i$ in $V(x)[h]$ do not repeat, then
$Z(x)$ is uniquely determined and is symmetric.

\begin{exa}
A ``middle matrix" representation $g=2$
\end{exa}
\begin{tabularx}{\linewidth}{X}
$3\,x_1h x_2^2hx_1 + hx_1x_2x_1h - h x_1 h x_2^2
-  x_2^2 h x_1 h
+
5\,x_1x_2hx_2hx_2x_1$ \\
\quad = $
\begin{pmatrix}
h \\ hx_1\\ hx_2x_1\\ hx_2^2\\
\end{pmatrix}^T
\begin{pmatrix}
x_1x_2x_1 & 0 & 0 & - x_1 \\
0 & 3\,x_2^2 & 0 & 0 \\
0 & 0 & 5\,x_2 & 0 \\
-x_1 & 0 & 0 & 0
\end{pmatrix}
\begin{pmatrix}
h \\ hx_1\\ hx_2x_1\\ hx_2^2\\
\end{pmatrix}.
$
\end{tabularx}

\subsubsection{Positivity of  $q$ vs. Positivity of its Middle Matrix}

A key fact is that positivity of $q$ is equivalent to positivity
of its middle matrix in the following sense.

\begin{lemma}
 \label{lem:datmostsublemma}
    Suppose $q(x)[h]$ a symmetric noncommutative polynomial in
    noncommuting variables pure quadratic in $h$ and $Z(x)$ is its
    middle matrix.  If $X \in \gtupn$ and $Z(X)\succeq 0$, then
    $q(X)[H] \succeq 0$ for all $H\in\gtupn$.

    Conversely, if $q(x)[h]$ is matrix positive; i.e., $q(X)[H]\succeq
    0,$ for every $n$ and $X,H\in\gtupn$ in the (non empty) positivity
    domain $\{X: \ f(X) \succ 0 \}$ of some polynomial $f$, then for
    each $n$ and $X \in \gtupn$, we have $Z(X)\succ 0$ on $\{X: \ f(X)
    \succeq 0 \}$.
\end{lemma}

\proof
  The first statement is evident. The converse is
  proved in   \cite{CHSY03}  in Lemma 9.5 and
Theorem 10.10 in  \cite{CHSY03}.
 for a cleaner proofs see \cite{HMVjfa} in
 particular Proposition 6.1. \qed

\subsection{The Zeroes Lemma}

The following is useful in our analysis of subharmonics.
\begin{lemma}
Let $S$
be any $N\times N$ symmetric matrix with entries in $\RRx$.
If there exists some diagonal entry $S_{ii} = 0$ and
corresponding off-diagonal entries \mbox{$S_{ij} = S^T_{ji} \neq 0$},
then  $S$ is not matrix-positive semidefinite.
\end{lemma}

\proof \ \
Let $e_i$ and $e_j$ be standard basis vectors for $R^N$ \ (i.e.
$e_i^TA\,e_j = a_{ij}$)
\  and define\\  $v := \beta_1e_i+\beta_2 e_j$ where $\beta_1,\,
\beta_2  \in \RR$.
Then,
\begin{equation*}
\qquad v^TS(x)v = ((S_{ij}\,+\,S_{ji})\,\beta_1
+ S_{jj}\,\beta_2 )\,\beta_2  =
(2\,S_{ij}\,\beta_1 + S_{jj}\,\beta_2 )\,\beta_2
\end{equation*}
Given $\beta_2 >0$, we can choose $\beta_1$ such that
\begin{equation}
2\,\beta_1\,S_{ij}(X) \beta_2+ \beta_2S_{jj}(X) \,\beta_2
\eeq\end{equation}
is neither a positive nor negative matrix.\qed

This lemma is  useful when applied to our matrix representation of the
Laplacian of a symmetric NC polynomial .

\subsection{The Laplacian of a Degree 4 Polynomial}
\label{sec:main:2b}
We begin with a parameterization
the set of degree 4 homogeneous polynomials in symmetric free variables
\begin{align*}
p = A_1\,&x_1^4\ncm x_1^2 + A_2 (x_1^3\ncm x_2 + x_2\ncm x_1^3)
 +A_3\,(x_1^2\ncm x_2\ncm x_1 + x_1\ncm x_2\ncm x_1^2)
+A_4\,(x_1^2\ncm x_2^2
+ x_2^2\ncm x_1^2)\\
&+A_5\,(x_1\ncm x_2\ncm x_1\ncm x_2 + x_2\ncm x_1\ncm x_2\ncm x_1) +
A_6\,x_1\ncm x_2^2\ncm x_1
+A_7\,(x_1\ncm x_2^3 + x_2^3\ncm x_1) + A_8\,x_2\ncm x_1^2\ncm x_2 \\
&+A_9\,(x_2\ncm x_1\ncm x_2^2 + x_2^2\ncm x_1\ncm x_2) + A_{10}\,x_2^4.
\end{align*}
We calculate the Laplacian of $p$:
\begin{align*}
2\,A_1\,&(h^2\ncm x_1^2 + h\ncm x_1\ncm h\ncm x_1 + h\ncm x_1^2\ncm h
+ x_1\ncm h^2\ncm x_1 + x_1\ncm h\ncm x_1\ncm h + x_1^2\ncm h^2) \\
+\,2\,&A_2\,(h^2\ncm x_1\ncm x_2 + h\ncm x_1\ncm h\ncm x_2 + x_1\ncm h^2\ncm
x_2
+ x_2\ncm h^2\ncm x_1 + x_2\ncm h\ncm x_1\ncm h + x_2\ncm x_1\ncm h^2) \\
+\,2\,&A_3\,(h^2\ncm x_2\ncm x_1 + h\ncm x_1\ncm x_2\ncm h + h\ncm x_2\ncm
h\ncm x_1
+ h\ncm x_2\ncm x_1\ncm h + x_1\ncm h\ncm x_2\ncm h + x_1\ncm x_2\ncm h^2)
\\
+\,2\,&A_4\,(h^2\ncm x_1^2 + h^2\ncm x_2^2 + x_1^2\ncm h^2 + x_2^2\ncm h^2)
\\
+\,2\,&A_5\,(h\ncm x_1\ncm h\ncm x_1 + h\ncm x_2\ncm h\ncm x_2 + x_1\ncm
h\ncm x_1\ncm h + x_2\ncm h\ncm x_2\ncm h)
+\,2\,A_6\,(h\ncm x_2^2\ncm h + x_1\ncm h^2\ncm x_1) \\
+\,2\,&A_7\,(h^2\ncm x_2\ncm x_1 + h\ncm x_2\ncm h\ncm x_1 + x_1\ncm h^2\ncm
x_2
+ x_1\ncm h\ncm x_2\ncm h + x_1\ncm x_2\ncm h^2 + x_2\ncm h^2\ncm x_1)
+\,2\,A_8\,(h\ncm x_1^2\ncm h + x_2\ncm h^2\ncm x_2) \\
+\,2\,&A_9\,(h^2\ncm x_1\ncm x_2 + h\ncm x_1\ncm h\ncm x_2 + h\ncm x_1\ncm
x_2\ncm h
+ h\ncm x_2\ncm x_1\ncm h + x_2\ncm h\ncm x_1\ncm h + x_2\ncm x_1\ncm h^2)
\\
+\,2\,&A_{10}\,(h^2\ncm x_2^2 + h\ncm x_2\ncm h\ncm x_2 + h\ncm x_2^2\ncm h
+ x_2\ncm h^2\ncm x_2 + x_2\ncm h\ncm x_2\ncm h + x_2^2\ncm h^2).
\end{align*}

 The directional Laplacian is quadratic in $h$,
and so can be represented by border vector
\begin{equation}
V(x)[h]^T=
\begin{pmatrix}
h & \ x_1\,h & \ x_2\,h & \ x_1^2\,h &
\ x_1\,x_2\,h & \ x_2\,x_1\,h & \ x_2^2h
\end{pmatrix}^T
\eeq\end{equation}
and middle matrix $Z(x)$ which is
\begin{equation}
\begin{pmatrix}
     ^{(A_{1}\,+\,A_{8})\,x_1^2\,+\,(A_{6}\,+\,A_{10})\,x_2^2}
     _{+\,(A_{3}\,+\,A_{9})\,(x_1\ncm x_2\,+\,x_2\ncm x_1)\,} &
     \!^{(A_{1}\,+\,A_{5})\,x_1}_{\,+\,(A_{3}\,+\,A_{7})\,x_2}&
     \!^{(A_{2}\,+\,A_{9})\,x_1}_{\,+\,(A_{5}\,+\,A_{10})\,x_2}&
     \!\!^{A_{1}\!+\!A_{4}} &
     ^{A_{3}\!+\!A_{7}} &
     ^{A_{2}\!+\!A_{9}} &
     ^{\!A_{4}\!+\!A_{10}} \\ \\
     ^{(A_{1}\,+\,A_{5})\,x_1}_{\,+\,(A_{3}\,+\,A_{7})\,x_2}&
     ^{A_{1}\,+\,A_{6}} &
     ^{A_{2}\,+\,A_{7}} &0 &0 &0 &0\\ \\
     ^{(A_{2}\,+\,A_{9})\,x_1}_{\,+\,(A_{5}\,+\,A_{10})\,x_2}&
     ^{A_{2}\,+\,A_{7}} &
     ^{A_{8}\,+\,A_{10}} &0 &0 &0 &0\\ \\
     _{A_{1}\,+\,A_{4}} &0 &0 &{\bf 0} &0 &0 &0\\ \\
     _{A_{3}\,+\,A_{7}} &0 &0 &0 &{\bf 0} &0 &0\\ \\
     _{A_{2}\,+\,A_{9}} &0 &0 &0 &0 &{\bf 0} &0\\ \\
     _{A_{4}\,+\,A_{10}} &0 &0 &0 &0 &0 &{\bf 0}
\end{pmatrix}.
\eeq\end{equation}
Assume that $Z(X)$ is a positive semidefinite matrix for
$X\in(\mathbb{R}^{n\times n}_{sym})^g$.
By the Zeroes Lemma, the zeroes on the last four diagonals
force all entries
in the last four rows or columns to be zero, that is:
\begin{equation}
A_4 = -A_1,\quad A_{10} = A_1,\quad A_9 = -A_2,\quad A_7 = -A_3.
\end{equation}

Applying these conditions to the matrix above,
 and ignoring the rows and columns which are zero, we have:
\begin{equation*}
\begin{pmatrix}
            _{+\,(A_2\, -\, A_3)(x_2\ncm\, x_1\,-\,x_1\ncm\, x_2)}^{(A_1\, +\,
A_8)\, x_1^2\,+\,(A_1\, +\, A_6)\, x_2^2} &
            ^{(A_1 + A_5) x_1} & ^{(A_1 + A_5) x_2} \\
            _{(A_1 + A_5) x_1} & _{A_1 + A_6} & _{A_2 - A_3} \\
            _{(A_1 + A_5) x_2} & _{A_2 - A_3} & _{A_1 + A_8}
\end{pmatrix}.
\end{equation*}
This matrix can be simplified by substitution
of reoccurring pairs by single
letters:
\begin{gather*}
G = A_1 + A_6,\quad H = A_1 + A_8\quad J = A_2 - A_3,\quad K = A_1 + A_5.
\end{gather*}
to obtain
\begin{equation*}
\begin{pmatrix}
H x_1^2 - J (x_1\ncm x_2 + x_2\ncm x_1) + G x_2^2 & K x_1 & K x_2 \\
K x_1 & G & J\\
K x_2 & J & H
\end{pmatrix}.
\end{equation*}
We now  find its noncommutative $LDL^T$ (Cholesky) decomposition to
have $D$ term equal to
\begin{equation*}
\begin{pmatrix}
G& 0& 0 \\
0& H-\frac{J^2}{G}& 0 \\
0& 0& H x_1^2 + J (x_1\ncm x_2 + x_2\ncm x_1) +  G x_2^2- \frac{K^2
x_1^2}{G} - \frac{(\frac{J K x_1}{G} + K x_2)\ncm (\frac{J K x_1}{G} + K
x_2)}{H - \frac{J^2}{G}}
\end{pmatrix}.
\end{equation*}
A reference is \cite{CHSY03}  which describes the NCAlgebra
command, NCLDUDecomposition,  we used to do this.

We see there are three inequalities,
which must be satisfied for $Z(X)$ to be positive semidefinite.
\begin{equation*}
G > 0,\quad  H-\frac{J^2}{G} > 0,
\end{equation*}
$$
\quad H X_1^2 + J (X_1\ncm X_2 + X_2\ncm X_1) +  G X_2^2 - \frac{K^2
X_1^2}{G} - \frac{(\frac{J K X_1}{G} + K X_2)\ncm (\frac{J K X_1}{G} + K
X_2)}{H - \frac{J^2}{G}} > 0.
$$
The last condition is purely quadratic in $X_1$ and $X_2$,
and therefore has a middle matrix representation which we
compute to be
\begin{equation*}
\begin{pmatrix}
G-\frac{H^2K^2}{(G-\frac{H^2}{J})J^2} - \frac{K^2}{J} &
H+\frac{HK^2}{(G-\frac{H^2}{J})J} \\
H+\frac{HK^2}{(G-\frac{H^2}{J})J} &
J-\frac{K^2}{G-\frac{H^2}{J}}
\end{pmatrix}.
\end{equation*}

Again
we perform the $LDL^T$ decomposition:
\begin{equation*}
\begin{pmatrix}
G - \frac{K^2}{H - \frac{J^2}{JG}}& 0 \\
0& H - \frac{J^2K^2}{(H-\frac{J^2}{G})G^2} - \frac{K^2}{G} -
\frac{\bigl(J-\frac{J K^2}{(H-\frac{J^2}{G})G}\bigr)^2}{G - \frac{K^2}{H -
\frac{J^2}{G}}}
\end{pmatrix}.
\end{equation*}
Although the inequality
\begin{equation}
H - \frac{J^2K^2}{(H-\frac{J^2}{G})G^2} - \frac{K^2}{G} -
\frac{\bigl(J-\frac{J K^2}{(H-\frac{J^2}{G})G}\bigr)^2}{G - \frac{K^2}{H -
\frac{J^2}{G}}} > 0
\label{eq:comp}
\end{equation}
is quite complicated, we can simplify it some by multiplying it by
expressions which are known to be positive, such as:
\begin{equation*}
G,\quad H - \frac{J^2}{G},\quad \text{and}\quad G - \frac{K^2}{H -
\frac{J^2}{G}}
\end{equation*}
which we encountered earlier. This gives a polynomial inequality equivalent
to (\ref{eq:comp}), which, after some simplification, gives us:
\begin{equation*}
(G H - J^2 - K^2)^2 > 0.
\end{equation*}
Which, considering only the case of all real coefficients, is rather
vacuous, informing us only that $GH - J^2 - K^2 \neq 0$.

Bringing all our inequalities together
(simplifying each as we did above),
we obtain
\begin{equation*}
(I)\; G > 0,\qquad  (II)\; G H > J^2,\qquad (III)\; G H > J^2 + K^2,\qquad
(IV)\; GH \neq H^2 + K^2.
\end{equation*}
Notice  (III) implies (II) and (IV),
 thus reducing to $(I)$ and $(II)$.
 Therefore, we conclude that the set
of polynomials making the Laplacian
matrix ``positive''  is exactly those of the form:
\begin{eqnarray}
f&=&A_1(x_1^4-x_1^2\ncm x_2^2-x_2^2\ncm x_1^2+x_2^4)
+A_2(x_1^3\ncm x_2+x_2\ncm x_1^3-x_2\ncm x_1\ncm x_2^2-x_2^2
\ncm x_1\ncm x_2)\\
&&+A_3(x_1^2\ncm x_2\ncm x_1 + x_1\ncm x_2\ncm x_1^2 - x_1
\ncm x_2^3-x_2^3\ncm x_1 )
+A_5(x_1\ncm x_2\ncm x_1\ncm x_2+x_2\ncm x_1\ncm x_2\ncm x_1)\notag\\
&&+A_6\,x_1\ncm x_2^2\ncm x_1 + A_8 x_2\ncm x_1^2\ncm x_2 \notag
\end{eqnarray}
with coefficients satisfying the inequalities:
\begin{align}
&(III)\implies(A_1 + A_8)(A_1 + A_6) >
(A_3 - A_2)^2 + (A_1 + A_5)^2 \text{ and }\\
&(I)\implies\quad A_1 + A_8> 0\quad (\text{or, equivalently}
\quad A_1 + A_6 >
0).\end{align}
For neatness, and to more clearly see the dimension of the
space of subharmonics,we set
$B_1=A_1, B_2=A_2, B_3=A_3, B_4=A_5, B_5=A_6, B_6=A_8$.
\qed

\section{Uniqueness Proofs}
\label{sec:unique}
Now we set about to prove that
the list of subharmonic and harmonic
polynomials in Theorem \ref{thm:main}
is complete. We do  this, as is required, only
for two variables but in the course of
our proof
we discover some promising recursions valid
in any number of variables.

\subsection{Even Degree Homogeneous $p$}
Given $ 0<m<d$ a noncommutative polynomial $p$  of degree $d$
decompose it as
\begin{equation}
\label{eq:defneighbor}
p = \sum_{|t|= m} x^{t} p_{t}(x)
+ \Lambda
\end{equation}
where $deg \ \Lambda  <m$.
Call the polynomial $p_{t}(x)$
the {\bf right neighbor of $x^{t}$}

\begin{lemma}
\label{lem:subToharm}
If $p$ is harmonic in any number of variables
consider the right neighbor representation
of $p$ for any $m$;
the right neighbor $p_{t}$ of
 each monomial $x^{t}$ of degree $m$ is harmonic,
 that is, Lap($p_{t})=0$.

If $p$ is subharmonic in any number of variables,
if $p$ is homogeneous of degree $d$
then the right neighbor $p_{t}$ of
 each monomial $x^{t}$ of degree $\dd2$ is harmonic,
 that is, Lap($p_{t})=0$.
\end{lemma}

\proof\
Apply the Laplacian to
the right neighbor decomposition (\ref{eq:defneighbor}) of
$p$ and get from
the product rule for the Laplacian (Lemma \ref{lem:prodRulap}):
\vspace{1ex}

$\ncl{p}{h}=$
\begin{equation}
\label{eq:Lapp1}
 \sum_{|t|= m} x^{t}  \ncl{p_{t}(x)}{h}
\end{equation}
\begin{equation}
\label{eq:Lapp2}
+  \sum_{|t|= m}  \ncl{x^{t}}{h} p_{t}(x)
\end{equation}
\begin{equation}
\label{eq:Lapp3}
+ 2\sum_{|t|= m} \dird(x^{t})[h] \dird(p_{t})[h]
\end{equation}
$$ + \ncl{\Lambda}{h}. $$

Suppose $\ncl{p}{h}=0$.
We shall now show that polynomial
(\ref{eq:Lapp1}) is 0, (\ref{eq:Lapp2}) is 0, and
(\ref{eq:Lapp3}) is 0.
All  terms of the
polynomials (\ref{eq:Lapp1}), (\ref{eq:Lapp2}), and
(\ref{eq:Lapp3})  have  degree at least $m$,
while $deg \ \Lambda <m$.
Since the Laplacian of a polynomial respects  degree,
we have
$\lap(\Lambda)=0$.
Next factor a given degree $\geq m$
monomial $r$ into its {\it $m-$front and back:}
namely, $r=r_f r_b$ where $r_f$ has degree $m$.
Consider the polynomial (\ref{eq:Lapp1}):
the $m-$back of each monomial in it contains
two $h$'s.
Likewise the $m-$back of each monomial in
(\ref{eq:Lapp2}) and (\ref{eq:Lapp3})
contains no $h$'s and one $h$ respectively.
Thus polynomials (\ref{eq:Lapp1}), (\ref{eq:Lapp2}), and
(\ref{eq:Lapp3}) contain no monomials  which
 cancel and since their sum is zero
they must be zero.
From (\ref{eq:Lapp1}) is 0 we immediately get
$Lap(p_t)=0$.
This proves  the first part of the lemma.

Now to the subharmonic part.
That $\ncl{p}{h}$ is matrix positive
implies that it is a sum of squares:
\begin{equation}
\label{eq:LapSoS}
\lap(p) = \sum_j L_j^T(x)[h] L_j(x)[h]
\end{equation}
First observe that each $L_j$ is linear in $h$.
This is true since the highest degree in $h$ monomial
$\lambda(x)[h]$
of  $ L_j(x)[h]$ contributes a $\lambda(x)[h]^T\lambda(x)[h]$
to  $L_j^T(x)[h] L_j(x)[h]$ monomial
which holds because its coefficient is positive and can not be cancelled out;
likewise $\lambda(x)[h]^T\lambda(x)[h]$ appears in $\lap(p)$.
Thus the monomial $\lambda(x)[h]$ has degree one in $h$.

Because of equation (\ref{eq:LapSoS})
we can refer to each term of $\lap(p)$
as having a first half and second
half; each half has degree $\dd2$.
Also every term
of $\lap(p) $ has an $h$ in its first half
 and also in its second half.
 However, if $m= \dd2$
all terms in (\ref{eq:Lapp1}) have two $h$'s in their
second half and none in their first half.
This contradicts the previous sentence;
thus equation (\ref{eq:Lapp1}) is 0.
Since we have factored out $x^t$ in the representation
(\ref{eq:Lapp1}), their coefficients $Lap(p_t)$ are 0
for all $|t|=\dd2$.
\qed

\subsubsection{Homogeneous Subharmonics
are Sums of Products of Harmonics}\label{sec:EvenSumHarms}

In this section we prove under weak
hypotheses that homogeneous subharmonics
are sums of products of harmonics.
A subharmonic polynomial of odd degree is harmonic,
so is the product of itself and 1.
Thus we restrict to even degree and prove the following.

\begin{prop}
\label{prop:EvenSubharmSumProdHarm}
Assume the harmonic polynomials
homogeneous of degree $\dd2$ are the span of $\h_1, \cdots, \h_k$.
Assume there is a monomial $w_j$ in $\h_j$ which does not occur
in the other $\h_1,  \cdots, \h_k$.
\\
If $p$ is subharmonic homogeneous of even degree $d$, then
it has the form
\begin{equation}\label{eq:repp}
p = \sum_{i,j=1}^k \phi_{ij} \h_i \h_j
\end{equation}
where each $\phi_{ij}$ is a real number.
Note further that for symmetric $p$
we may take $\phi_{ij}=\phi_{ji}$.
Let  $ \cS$ denote the span  of these
 symmetric subharmonics.

This implies that  $\cS$
is a space of at most dimension $\frac{k(k+1)}{2}$.
For example, in two variables there are two  independent
homogeneous harmonic polynomials of degree other than 2, so
 $dim \ \cS$ is at most 3 for all $d \neq 4$.
For $d=4$ we have $dim \ \cS \leq 6$.
\end{prop}

\proof \ \
Assume $p$ is subharmonic homogeneous of degree $d$.
Write down its right neighbor representation
with $m=\dd2$ and
use Lemma \ref{lem:subToharm} to get $Lap( p_{t})=0$ for $|t|= \dd2$.
Thus
$$p_{t}= \sum_j^k \mu_j(t) \h_j$$
for some numbers $\mu_j(t)$.
Plug this into the decomposition
\begin{equation}
p = \sum_{|t|= \dd2} x^{t} p_{t}(x)
 \ =\ \sum_{ |t|= \dd2} \sum_j \mu_j(t) x^{t} \h_j
\end{equation}
to get
$$
p =\sum_j \left(\sum_{|t|= \dd2}\mu_j(t)x^{t}\right)\h_j
=\sum_j^k p^j(x) \h_j.
$$
Now make a left neighbor decomposition of $p$
which by the definition of the monomial $w_1$ has the form
$$ p^1(x) w_1 +  G$$
where all terms of $G$ are  without $w_1$ on the right.
The left neighbor version of
Lemma \ref{lem:subToharm} implies $ Lap(p^1(x) )=0$.
Likewise each $p^j(x)$ is harmonic of degree $\dd2$.
This proves representation (\ref{eq:repp}) for $p$.
\qed

Next we prove   our representation of subharmonics stated
in the introduction as Theorem \ref{thm:SubsToSquaresOfHarms}.

\noindent
{\bf Proof  of  Theorem \ref{thm:SubsToSquaresOfHarms}} \\
Now suppose $p$ is symmetric.
 Proposition \ref{prop:EvenSubharmSumProdHarm}  says we can
 represent $p$ as in equation (\ref{eq:repp}).
 Note if $u$ is harmonic then $u^T$
is harmonic and relabel and possibly expand (by taking transposes)
the set $\gamma_1, \dots, \gamma_k$ as
$$s_1, \dots, s_\alpha, u_1, \dots, u_\beta, u_1^T, \dots, u_\beta^T$$
where the $s_i$ are symmetric polynomials.
Set $\Psi:= \{ \tilde \phi_{ij} \}_{i,j=1}^{\alpha+2\beta}$
where  $\tilde \phi_{ij}= \phi_{ij}$
for   $i,j$  corresponding to an original $\gamma_\ell$
and 0 otherwise.
Now let
$s=
\begin{pmatrix}
s_1\\ \vdots \\ s_\alpha
\end{pmatrix}$,
$u=\begin{pmatrix}
u_1\\ \vdots \\ u_\beta
\end{pmatrix}$
and
$v=
\begin{pmatrix}
u_1^T\\ \vdots \\ u_\beta^T
\end{pmatrix}
$.
Then
$p=\begin{pmatrix}
s \\ u \\ v
\end{pmatrix}^T
\Psi
\begin{pmatrix}
s \\ u \\ v
\end{pmatrix}$
and
$$
p = \dfrac{p+p^T}{2} =
\begin{pmatrix}
s \\ u \\ v
\end{pmatrix}^T
\Phi
\begin{pmatrix}
s \\ u \\ v
\end{pmatrix}
\qquad
$$
 where $\Phi
= \dfrac{\Psi + \Psi^T}{2}$,
a symmetric matrix as required.

Decompose the symmetric matrix $\Phi$ as $ \Phi= N J N^T$ where $J$
is a diagonal  matrix with $\pm 1$ or 0 on the diagonal
 and $N$ has real
numbers as entries.
Now, let us put $R=N^T\begin{pmatrix}
s \\ u \\ v
\end{pmatrix}$. Then
$$
p=
\begin{pmatrix}s \\ u \\ v\end{pmatrix}^T \Phi \begin{pmatrix}s \\ u \\ v\end{pmatrix}
=
\begin{pmatrix}
s \\ u \\ v
\end{pmatrix}^T N J N^T\begin{pmatrix}
s \\ u \\ v
\end{pmatrix}
= R^T J R
= \sum_i c_i R_i^T R_i.
$$
where $c_j$  is $\pm 1$ or 0.
The $s_i,\ u_i,\ u_i^T$ are harmonic, so their linear combinations
 $R_i$ are harmonic.
\qed

An appealing, easily proved, formula is
$$
\lap(p)
= 2 \sum_j^g \dird[R, x_j, h]^T J \dird[R,x_j , h].
$$
Clearly, if the matrix $ \Phi:= \{\phi_{ij} \}_{i,j=1}^{\alpha+2\beta}$ is
positive semidefinite (or equivalently $J$ has nonnegative entries),
$\lap(p)$
will be positive, so then $p$ is subharmonic.
Also we get even degree harmonics
are sums and differences of squares
of harmonics.
It is not clear which differences make $p$ harmonic or subharmonic.

However, we conjecture\\
{\it
A homogeneous symmetric subharmonic polynomial $p$ of
 even degree $d$ is a finite sum
 $$ p=  \sum_i^{finite}  R_i^T R_i + \sum_\ell^{finite} H_\ell $$
for some homogeneous harmonic functions $R_i,H_\ell$ with
$R_i$ of degree $\dd2$ and $H_\ell$ of degree $d$.
 }

At this point, we have finished discussing subharmonics,
and will now turn our full attention to  harmonic polynomials.

\subsection{Uniqueness of Harmonics in Two Variables}
\subsubsection{Polynomials of degree Three and Larger}
\label{sec:main:1a}
\label{sec:allharmpf}

At this point, we have proved that there
are harmonic polynomials of arbitrary degree.
Working in two variables, we will now show that
the polynomials $\re\h^d$ and $\im\h^d$ span all of the harmonics.
This can be a helpful result, which, as yet, we have been
unable to show for any higher number of variables.
 In fact, McAllaster  has found, experimentally,
 that in three variables, the size of the basis
 of harmonic polynomials increases on the order of $d^2$
 (See \cite{harmonic}).

\begin{prop}\label{prop:allharmprop}
Let $\gamma=x_1+ix_2$, and $\mathcal{B}_d=\{\re\gamma^d, \im\gamma^d\}$.
 Then $\mathcal{B}_d$ forms a basis for all harmonic polynomials
which are homogeneous of degree $d$ for any $d>2$.
\end{prop}

To prove the proposition we need two lemmas.
\begin{lemma}\label{lem:deg3}
In degree three, there are two linearly independent homogeneous
harmonic polynomials whose span is all harmonic polynomials which
are homogeneous of degree three.
\end{lemma}

\begin{lemma}\label{lem:lowerdeg}
Let $\beta(x_1,x_2)$ be harmonic and homogeneous of degree $d$. Then we
may uniquely represent $\beta$ as $\beta(x_1,x_2)=x_1 f(x_1,x_2)+x_2 g(x_1,x_2)$,
where $f$ and $g$ are harmonic and homogeneous of degree $d-1$ and $\ncder{f(x_1,x_2)}{x_1}{h}=-\ncder{g(x_1,x_2)}{x_2}{h}.$
\end{lemma}

\begin{proof} \textbf{(Lemma \ref{lem:deg3})} \ \
Every homogeneous polynomial of degree three has the form
$$a_1x_1^3+a_2x_1^2x_2+a_3x_1x_2x_1+a_4x_1x_2^2
+a_5x_2x_1^2+a_6x_2x_1x_2+a_7x_2^2x_1+a_8x_2^3$$
and the Laplacian of this is
$$a_1h^2x_1+a_7h^2x_1+a_2h^2x_2+a_8h^2x_2+a_1x_1h^2+a_4x_1h^2$$
$${}+a_5x_2h^2+a_8x_2h^2+a_1hx_1h+a_6hx_1h+a_3hx_2h+a_8hx_2h$$
$$=(a_1+a_7)h^2x_1+(a_2+a_8)h^2x_2+(a_1+a_4)x_1h^2
+(a_5+a_8)x_2h^2+(a_1+a_6)hx_1h+(a_3+a_8)hx_2h,$$
so if we want the polynomial to be harmonic,
we need each monomial of the Laplacian to be zero, so we need the equations
\(a_1+a_7=0,\ a_2+a_8=0,\ a_1+a_4=0,\ a_5+a_8=0,\ a_1+a_6=0,\ a_3+a_8=0\) to hold.
This amounts to having the vector
$(a_1,a_2,a_3,a_4,a_5,a_6,a_7,a_8)$  in the nullspace of the matrix
$$\left(
\begin{matrix}
1&0&0&0&0&0&1&0\\
0&1&0&0&0&0&0&1\\
1&0&0&1&0&0&0&0\\
0&0&0&0&1&0&0&1\\
1&0&0&0&0&1&0&0\\
0&0&1&0&0&0&0&1\\
\end{matrix}\right),
$$
which has the basis $\{(0,-1,-1,0,-1,0,0,1), (-1,0,0,1,0,1,1,0)\}$,
corresponding to polynomials
$x_2^3-x_1^2x_2-x_2x_1^2-x_1x_2x_1$ and
$-x_1^3+x_1x_2^2+x_2^2x_1+x_2x_1x_2$.
Hence there are exactly two linearly independent harmonic polynomials
which are homogeneous of degree three.
\end{proof}

\begin{proof} \textbf{(Lemma \ref{lem:lowerdeg})}
Now, suppose that we are given a polynomial function $\beta(x_1,x_2)$
 which is harmonic, homogeneous, and of degree $d$.
 Then, every monomial of $\beta$ begins with either $x_1$ or $x_2$,
 so we may uniquely represent $\beta$ by
the neighbor decomposition $\beta(x_1,x_2)=x_1 f(x_1,x_2)+x_2 g(x_1,x_2)$.
Now, by Lemma \ref{lem:subToharm}, we know that $f$ and $g$ are harmonic,
and using the product rule for the Laplacian, we find:
\begin{eqnarray*}
\ncl{\beta}{h}
&=& \ncl{x_1 f(x_1,x_2)+x_2 g(x_1,x_2)}{h}\\
&=&(\ncl{x_1}{h}f(x_1,x_2)+x_1\ncl{f(x_1,x_2)}{h} \
+ \ 2(\ncder{x_1}{x_1}{h}\ncder{f(x_1,x_2)}{x_1}{h}\\
& &+\ncder{x_1}{x_2}{h}\ncder{f(x_1,x_2)}{x_2}{h}))
 +(\ncl{x_2}{h}g(x_1,x_2)+x_2\ncl{g(x_1,x_2)}{h}\\
& &+2(\ncder{x_2}{x_1}{h}\ncder{g(x_1,x_2)}{x_1}{h}
+\ncder{x_2}{x_2}{h}\ncder{g(x_1,x_2)}{x_2}{h}))\\
&=&2h(\ncder{f(x_1,x_2)}{x_1}{h}+\ncder{g(x_1,x_2)}{x_2}{h})\\
\end{eqnarray*}
and since $\beta$ is harmonic, this is zero, which gives:
\[
0=2h(\ncder{f(x_1,x_2)}{x_1}{h}+\ncder{g(x_1,x_2)}{x_2}{h}).
\]
or more specifically,
\begin{equation}\label{eqn:zerocond}
\ncder{f(x_1,x_2)}{x_1}{h}=-\ncder{g(x_1,x_2)}{x_2}{h}.
\end{equation}
\end{proof}

\begin{proof} \textbf{(Prop. \ref{prop:allharmprop})}
First of all, we show that
$\re\gamma^d$ and $\im\gamma^d$
are linearly independent:

Suppose they are linearly dependent. Then $a\re\h^d=b\im\h^d$, where
$a\ne 0$, $b\ne 0$, $a,b\in\RR$. But then, $\h^d=(x_1+ix_2)(R+iI)$,
where $R=\re\h^{d-1}$, $I=\im\h^{d-1}$, so that $a(x_1R-x_2I)=b(x_1I+x_2R)$.
Now, we equate the terms starting with $x_1$ and $x_2$, respectively, to get
that $aR=bI$ and $-aI=bR$. Then, we get that $R=(b/a)I$ and $I=(-b/a)R$ from
the first and second equations, repectively. Puting this together, we get
$R=(b/a)I=(b/a)(-b/a)R$, so that cancelling $R$,
we get $b^2/a^2=-1$, or $b^2=-a^2$,
which can happen only if $a=b=0$, a contradiction.

Now, we are going to prove the proposition by induction.

First of all, Lemma \ref{lem:deg3} begins the induction.
To prove the rest of the proposition,
suppose that for degree $d-1 \geq 3$,
we know that there are exactly two linearly independent polynomials
which are harmonic and homogeneous.
Then for degree $d$,
we suppose that $\beta$ is harmonic and homogeneous. Then
$$
\beta(x_1,x_2)=x_1\varphi(x_1,x_2)+x_2\psi(x_1,x_2),
$$
where $\varphi$ and $\psi$ are both harmonic,
 and homogeneous of degree $d-1$ (Lemma \ref{lem:lowerdeg}).
 Then by the induction hypothesis
\[
\beta(x_1,x_2)=x_1(a_\varphi \re\gamma^{d-1}+b_\varphi
\im\gamma^{d-1})+x_2(a_\psi \re\gamma^{d-1}+b_\psi \im\gamma^{d-1}),
\]
but from Lemma \ref{lem:lowerdeg}, we know that we must have
$\ncder{\varphi}{x_1}{h}+\ncder{\psi}{x_2}{h}=0$
which is equivalent to saying that
\[
a_\varphi\ncder{\re\gamma^{d-1}}{x_1}{h}+b_\varphi
\ncder{\im\gamma^{d-1}}{x_1}{h}\\+a_\psi
\ncder{\re\gamma^{d-1}}{x_2}{h}+b_\psi\ncder{\im\gamma^{d-1}}{x_2}{h}=0.
\]
Now, by applying the identities for the derivatives of
 $\re\gamma^d$ and $\im\gamma^d$ (see Lemma \ref{lem:sym}),
  we get the following:
\[
a_\varphi\ncder{\re\gamma^{d-1}}{x_1}{h}+b_\varphi
\ncder{\im\gamma^{d-1}}{x_1}{h}\\-a_\psi\ncder{\im\gamma^{d-1}}{x_1}{h}
+b_\psi\ncder{\re\gamma^{d-1}}{x_1}{h}=0,
\]
so
\[
(a_\varphi+b_\psi)\ncder{\re\gamma^{d-1}}{x_1}{h}+(b_\varphi-a_\psi)\ncder{\im\gamma^{d-1}}{x_1}{h}=0,
\]
which gives that
\begin{eqnarray*}
 &0=&\ncder{(a_\varphi+b_\psi)\re\gamma^{d-1}
 +(b_\varphi-a_\psi)\im\gamma^{d-1}}{x_1}{h},
\end{eqnarray*}
but if the derivative of a function with respect to $x_1$ is zero,
that function must be a polynomial in $x_2$,
but we know that $\re\gamma^{d-1}$ and $\im\gamma^{d-1}$
are homogeneous of degree $d-1$,
so the function of $x_2$ can only be $cx_2^{d-1}$ for some constant $c$.
That is to say
\[(a_\varphi+b_\psi)\re\gamma^{d-1}+(b_\varphi-a_\psi)\im\gamma^{d-1}
=cx_2^{d-1}.\]
Now since $x_2^{d-1}$ is not harmonic, it is not in the span of
$\re\gamma^{d-1}$ and $\im\gamma^{d-1}$, so $c=0$. Therefore,
\[
0=(a_\varphi+b_\psi)\re\gamma^{d-1}+(b_\varphi-a_\psi)\im\gamma^{d-1},
\]
which means that $0=a_\varphi+b_\psi$ and $0=b_\varphi-a_\psi$, so
\[a_\varphi=-b_\psi\mbox{ and }a_\psi=b_\varphi.\]
Using this, we get
\begin{eqnarray*}
\beta(x_1,x_2)&=&x_1(a_\varphi \re\gamma^{d-1}+b_\varphi \im\gamma^{d-1})
+x_2(a_\psi \re\gamma^{d-1}+b_\psi \im\gamma^{d-1})\\
&=&x_1(a_\varphi\re\gamma^{d-1}+a_\psi\im\gamma^{d-1})
+x_2(a_\psi\re\gamma^{d-1}-a_\varphi\im\gamma^{d-1})\\
&=&a_\varphi x_1\re\gamma^{d-1}+a_\psi x_1\im\gamma^{d-1}
+a_\psi x_2\re\gamma^{d-1}-a_\varphi x_2\im\gamma^{d-1}\\
&=&a_\varphi(x_1\re\gamma^{d-1}-x_2\im\gamma^{d-1})
+a_\psi(x_1\im\gamma^{d-1}+x_2\re\gamma^{d-1})\\
&=&a_\varphi\re\gamma^{d}+a_\psi\im\gamma^{d},
\end{eqnarray*}
which implies that $\beta$ is linearly dependent upon
$\re\gamma^d$ and $\im\gamma^d$. Hence $\re\gamma^d$ and $\im\gamma^d$
form a basis for all of the harmonic polynomials
which are homogeneous of degree $d$.

\end{proof}

\subsubsection{Degree Two Polynomials}
\label{sec:main:2c}\label{sec:main:1b}
The polynomials of degree two are a special case.
This is because some terms of polynomials will vanish when
the Laplacian is taken.
Specifically, if we are given the general polynomial
$$
p=A_1x_1^2+A_2x_2^2+A_3x_1x_2+A_4x_2x_1,
$$
we find that the Laplacian is
$$
\lap(p)= A_1h^2+A_2h^2,
$$
meaning that the polynomial will be harmonic provided
that $A_1+A_2=0$ and subharmonic provided that $A_1+A_2\ge0$.
This is the one case where the harmonic polynomial is not symmetric.
Also, because the subharmonic polynomials are built up of
harmonic polynomials of one half the degree,
this means that it may be possible, in the degree four case alone,
to create nonsymmetric subharmonic polynomials.

\begin{remark}
\label{rem:deg6}
  We now show that this gives a 6 dimensional spanning set for
  the symmetric subharmonics of degree 4; denote these by $\cS_4$.
  We use Proposition \ref{prop:EvenSubharmSumProdHarm} which
  says $\cS_4$ is  spanned by symmetrized products of
  the basis $$ s=: x_1^2- x_2^2, u:= x_1x_2, u^T:= x_2x_1$$.
  Thus we obtain
  $$ s^2, su +u^T s, su^T +us, uu +u^T u^T, u^T u , u u^T.$$
 Note this is consistent with Theorem \ref{thm:main} part (2b)
 which implies the span of the degree 6 symmetric subharmonics
 has dimension.
\end{remark}

\subsection{Homogeneous Harmonics of Odd Degree}

The remainder of this section is not used in the rest of the paper,
but Proposition \ref{prop:oddSub} may be useful in further research
on harmonics in many variables.

What does the argument in Section \ref{sec:EvenSumHarms}
say about harmonics of odd degree?
As we have already stated in \S \ref{sec:main:3}, any subharmonic polynomial
of odd degree is required to be harmonic.

Given NC polynomial $p$ decompose it as
\begin{equation}
\label{eq:defneighboro}
p = \sum_{ |t|= \dm2} \sum_{i=1}^g x^{t} x_{i} p_{t,i}(x)
\end{equation}
and call the polynomial $p_{t,i}(x)$
the {\bf right neighbor of $x^{t} x_{i}$}.
Here we are assuming all terms of $p$
have degree $\dm2$.

Apply $Lap$ to
the right neighbor decomposition (\ref{eq:defneighboro}) of
harmonic $p$ and from the Laplacian  Product Rule   get \\
$0= \ncl{p}{h} =$
\begin{equation}
\label{eq:Lapp1oa}
 \sum_{ |t|= \dm2} \sum_{i=1}^g \ncl{x^{t} x_{i}}{h} p_{t,i}(x)
\end{equation}
\begin{equation}
\label{eq:Lapp2oa}
+ \sum_{ |t|= \dm2} \sum_{i=1}^g x^{t} x_{i} \ncl{p_{t,i}(x)}{h}
\end{equation}
\begin{equation}
\label{eq:Lapp3oa}
+ 2\sum_{ |t|= \dm2} \sum_{i=1}^g \sum_{j=1}^g \ncder{x^{t} x_{i}}{x_j}{h} \ncder{p_{t,i}(x)}{x_j}{h}
\end{equation}
which is
\begin{equation}
\label{eq:Lapp1o1}
\sum_{|t|= \dm2}\sum_{i=1}^g \left(\ncl{x^{t}}{h}x_{i}+x^{t}\ncl{x_{i}}{h}
+2\sum_{j=1}^g \ncder{x^{t}}{x_j}{h} \ncder{x_{i}}{x_j}{h}\right) p_{t,i}(x)
\end{equation}
\begin{equation}
\label{eq:Lapp2o2}
+ \sum_{ |t|= \dm2} \sum_{i=1}^g x^{t} x_{i} \ncl{p_{t,i}(x)}{h}
\end{equation}
\begin{equation}
\label{eq:Lapp3o3}
+ 2\sum_{ |t|= \dm2} \sum_{i=1}^g \sum_{j=1}^g \left(\ncder{x^{t}}{x_j}{h}x_i
+x^{t}\ncder{x_i}{x_j}{h}\right) \ncder{p_{t,i}(x)}{x_j}{h}
\end{equation}
Finally it becomes

\begin{equation}
\label{eq:Lapp1o}
\sum_{|t|= \dm2}\sum_{i=1}^g \left(\ncl{x^{t}}{h}x_{i}+2\ncder{x^{t}}{x_i}{h}h\right) p_{t,i}(x)
\end{equation}
\begin{equation}
\label{eq:Lapp2o}
+ \sum_{ |t|= \dm2} \sum_{i=1}^g x^{t} x_{i} \ncl{p_{t,i}(x)}{h}
\end{equation}
\begin{equation}
\label{eq:Lapp3o}
+2\sum_{ |t|= \dm2} \sum_{i}^g \left(x^{t}h\ncder{p_{t,i}(x)}{x_i}{h}
+  \sum_{j=1}^g \ncder{x^{t}}{x_j}{h}x_i\ncder{p_{t,i}(x)}{x_j}{h}\right),
\end{equation}
which must be 0.
The right half
of each monomial in (\ref{eq:Lapp1o}) contains no $h$'s,
while in $(\ref{eq:Lapp2o})$ and (\ref{eq:Lapp3o}) each right half does;
thus no term of (\ref{eq:Lapp1o}) can be cancelled.
 We conclude (\ref {eq:Lapp1o}) is 0.
 Similarly the right halves in (\ref{eq:Lapp2o}) are the only right halves
 monomials which contain two $h$'s
 and so cannot be cancelled.
 Thus (\ref{eq:Lapp2o}) is 0 and so we get

\centerline{
$\ncl{p_{t,i}}{h}=0$ for each $|t|= \dm2, i=1, \cdots, g$.
}

\noindent
Use
 \[p_{t,i}= \sum_j^k \mu_j(t,i) \h_j\]
as before.
Plug this into the decomposition
\begin{equation}
p = \sum_{ |t|= \dm2,i} x^{t} x_i p_{t,i}(x)
=\sum_{ |t|= \dm2,i}  \sum_j \mu_j(t,i) x^{t} x_i \h_j .
\end{equation}
to get

\[
p
=\sum_j \left(\sum_{|t|= \dm2,i} \mu_j(t,i) x^{t} x_i\right) \h_j
=\sum_j \left(\sum_{i}\left(\sum_{|t|= \dm2} \mu_j(t,i) x^{t}\right)x_i\right) \h_j
\]

\[
=\sum_{j=1}^{k} \left(\sum_{i}p^{j,i}(x)x_i\right) \h_j
\]

Now make a left neighbor decomposition of $p$
which  by the definition $\h_1$ has the form
\[\left(\sum_i p^{1,i}(x) x_i\right) \h_1  +  G\]
where all terms of $G$ are  without $\h_1$ on the right.
The left handed version of Lemma \ref{lem:subToharm}
implies $ Lap( p^{1,i} )=0$ for each $i=1, \cdots, g$.
\vspace{1ex}

We have proved the following:

\begin{prop}
\label{prop:oddSub}
Assume the harmonic polynomials
homogeneous of degree $\dm2$ are the span of $\h_1, \cdots, \h_k$.\\
Assume there is a monomial $w_j$ in $\h_j$ which does not occur
in the other $\h_1,  \cdots, \h_k$.
\\
If $p$ is subharmonic homogeneous of odd degree $d$,
then it is harmonic and has the form
\begin{equation}
p = \sum_{i=1}^g \sum_{m,j=1}^k \phi_{m i j} \h_m x_i \h_j
\end{equation}
where each $\phi_{m i j}$ is a number.
\end{prop}

\begin{ques}
What $\phi_{m i j}$ make it 0?
That is to say, what properties must $\phi_{m i j}$
satisfy in order for $p$ to be harmonic?
\end{ques}
We do a few calculations which might someday help with
this question.
Note the Laplacian of such a $p$ is:
$$
\ncl{p}{h} = \sum_{i=1}^g\sum_{m,j=1}^k \phi_{m i j} \ncl{\h_m x_i \h_j}{h}
$$
$$
= \sum_{i=1}^g\sum_{m,j=1}^k \phi_{m i j} \Bigl( \ncl{\h_m}{h}x_i \h_j
+\h_m\ncl{x_i}{h} \h_j + \h_mx_i  \ncl{\h_j}{h}
$$
$$
{}+ 2\sum_{l=1}^g ( \h_m \ncder{x_i}{x_l}{h} \ncder{\h_j}{x_l}{h}
+ \ncder{\h_m}{x_l}{h} x_i \ncder{\h_j}{x_l}{h}
+ \ncder{\h_m}{x_l}{h} \ncder{x_i}{x_l}{h} \h_j ) \Bigr)
$$
$$
= 2\sum_{i=1}^g\sum_{m,j=1}^k \phi_{m i j} \Bigl(\h_m h \ncder{\h_j}{x_i}{h} + \ncder{\h_m}{x_i}{h} h \h_j
+ \sum_{l=1}^g \ncder{\h_m}{x_l}{h} x_i \ncder{\h_j}{x_l}{h}\Bigr)
$$
As before cancellation cannot occur between terms
with right halves containing two $h$'s, one $h$ and no $h$'s.
Thus, Lap$[p,h]=0$ is equivalent to
$$
\sum\limits_{m=1}^k \gamma_m h
\sum\limits_{i=1}^g D\biggl[\sum\limits_{j=1}^k
\phi_{m i j} \gamma_j, x_i, h \biggl]
=0,
\quad \text{and} \quad
\sum\limits_{j=1}^k \biggl[\sum\limits_{i=1}^g
D \biggl[ \sum\limits_{m=1}^k \phi_{m i j} \gamma_m, x_i,h\biggl]
\biggl] h \gamma_j
=0 \qquad \text{and}
$$
$$
\sum\limits_{\ell=1}^g
\sum\limits_{i=1}^g \sum\limits_{j=1}^k
D\biggl[ \sum_m^k
\phi_{mij} \gamma_m, x_l, h \biggl] x_i D \biggl[ \gamma_j,
x_l,h \biggl]
=0.
$$

\section{Acknowledgments}
All authors were
partially supported by J.W. Helton's grants from
 the NSF and the Ford Motor Co.. and J. A. Hernandez
 was supported by a  McNair Fellowship.

Thanks to Nick Slinglend and John Shopple for help with computations.
Thanks to Professor Roger Howe for very helpful conversations about the
classical commutative analog of the noncommutative results here.
\newpage

\bibliographystyle{alpha}
\bibliography{subharmHHM}

\begin{thebibliography}{CHSY03}

\bibitem[CHSY03]{CHSY03}
Juan~F. Camino, J.~W. Helton, R~.E. Skelton, and Jieping Ye.
\newblock Matrix inequalities: a symbolic procedure to determine convexity
  automatically.
\newblock {\em Integral Equations Operator Theory}, 46(4):399--454, 2003.

\bibitem[Han97]{Han97}
Frank Hansen.
\newblock Operator convex functions of several variables.
\newblock {\em Publ. Res. Inst. Math. Sci.}, 33(3):443--463, 1997.

\bibitem[HM98]{geom}
J.~W. Helton and Orlando Merino.
\newblock {\em Sufficient Conditions for Optimization of Matrix Functions}.
\newblock CDC, 1998.
\newblock pp. 1--5.

\bibitem[HM04]{ConvPoly}
J.~W. Helton and Scott McCullough.
\newblock Convex noncommutative polynomials have degree two or less.
\newblock {\em Siam J. Matrix Anal. Appl}, 25(4):1124--1139, 2004.

\bibitem[HMV06]{HMVjfa}
J.~W. Helton, Scott~A. McCullough, and Victor Vinnikov.
\newblock Noncommutative convexity arises from linear matrix inequalities.
\newblock {\em J. Funct. Anal.}, 240(1):105--191, 2006.

\bibitem[HP06]{HPprept}
J.~W. Helton and M.~Putinar.
\newblock Positive polynomials in scalar and matrix variables, the spectral
  theorem and optimization.
\newblock To appear (see arXiv: http://arxiv.org/abs/math.FA/0612103), 2006.
\newblock p. 106.

\bibitem[HT92]{HoweTan}
Roger Howe and Eng-Chye Tan.
\newblock {\em Nonabelian harmonic analysis}.
\newblock Universitext. Springer-Verlag, New York, 1992.
\newblock Applications of ${\rm SL}(2,{\bf R})$.

\bibitem[HT06]{HanT06}
Frank Hansen and Jun Tomiyama.
\newblock Differential analysis of matrix convex functions.
\newblock {\em Linear Algebra and its Applications}, 2006.

\bibitem[McA04]{harmonic}
Daniel~P. McAllaster.
\newblock Homogeneous harmonic noncommutative polynomials from degree 3 to 7
  (includes a recurrence relation for higher degrees).
\newblock http://math.ucsd.edu/\~{}dmcallas/papers/, July 2004.

\bibitem[SV06]{AIMfree06}
Dimitri Shlyakhtenko and Dan Voiculescu.
\newblock Free analysis workshop summary: American institute of mathematics.
\newblock http://www.aimath.org/pastworkshops/freeanalysis.html, 2006.

\end{thebibliography}

\end{document}